\documentclass[11pt]{amsart}
\usepackage{amsmath}
\usepackage{amssymb, verbatim}
\usepackage{pstricks,pst-node,pst-plot}
\usepackage{comment}

\title{Log-canonical thresholds in real and complex dimension 2}
\author[T.C. Collins]{Tristan C. Collins}
  \email{tcollins@math.harvard.edu}
  \thanks{Work supported in part by NSF grant DMS-1506652}
  \address{Department of Mathematics, Harvard University, 1 Oxford Street, Cambridge, MA 02138}
    \dedicatory{Dedicated to J.-P. Demailly on the occasion of his 60th birthday.}

\theoremstyle{plain}
\newtheorem{thm}{Theorem}[section]
\newtheorem{prop}[thm]{Proposition}
\newtheorem{defn}[thm]{Definition}
\newtheorem{lem}[thm]{Lemma}

\newtheorem{conj}[thm]{Conjecture}

\newtheorem{nota}[thm]{Notation}

\theoremstyle{definition}

\newtheorem{rk}[thm]{Remark}

\numberwithin{equation}{section}

\newcommand{\del}{\partial}

\renewcommand{\leq}{\leqslant}
\renewcommand{\geq}{\geqslant}

\renewcommand{\epsilon}{\varepsilon}
\renewcommand{\phi}{\varphi}

\begin{document}
\maketitle

\begin{abstract}
We study the set of log-canonical thresholds (or critical integrability indices) of holomorphic (resp. real analytic) function germs in $\mathbb{C}^2$ (resp. $\mathbb{R}^2$).  In particular, we prove that the ascending chain condition holds, and that the positive accumulation points of decreasing sequences are precisely the integrability indices of holomorphic (resp. real analytic) functions in dimension $1$.  This gives a new proof of a theorem of Phong-Sturm.
\end{abstract}

\section{Introduction}


Let $f$ be a holomorphic or real analytic function defined in a neighbourhood of the origin in $\mathbb{C}^{n}$ or $\mathbb{R}^{n}$.  We define the critical integrability index, or log-canonical threshold, of $f$ at the origin to be
\[
c_0(f):= \sup \left\{ c >0 : \exists \epsilon >0 \text { such that } \int_{B_{\epsilon}(0)} |f|^{-c} < +\infty \right\}
\]

The number $c_0(f)$ is a measure of the order of vanishing of $f$ at the origin; indeed in dimension $1$ it is easy to see that $c_0(f)$ is proportional to the order of vanishing, while in higher dimension $c_0(f)$ is an important and subtle invariant of the set $\{f=0\}$.  In real analysis the critical integrability index plays an prominent role in the asymptotic analysis of oscillatory integral operators \cite{V, PSt, G}, and has recently appeared has an important invariant in statistics (see \cite{L} and the references therein).  In complex differential and algebraic geometry, the critical integrability index (and its generalizations to ideal sheaves) appear in geometric applications including existence of K\"ahler-Einstein metrics on Fano manifolds (where $c_0(f)$ is called the $\alpha$-invariant) \cite{Siu, TY, TY1, T, T1, DK, B}, and in the minimal model program (where $c_0(f)$ is called the log-canonical threshold) \cite{S, S1, Ko, Ko1}.

In their study of complex singularity exponents and applications to K\"ahler-Einstein metrics on Fano orbifolds, Demailly-Koll\'ar \cite{DK} proposed a series of remarkable conjectures regarding the numbers $c_0(f)$, including the following
\begin{conj}\label{conj: ACC}
The set
\[
\mathcal{C}(n) = \{ c_0(f) : f \in \mathcal{O}_{\mathbb{C}^n,0} \} 
\]
satisfies the ascending chain condition: precisely, any ascending sequence eventually stabilizes.
\end{conj}
A closely related, and well-known conjecture, called the Ascending Chain Condition (ACC) conjecture, appears in the algebraic geometry literature dating back to Shokurov \cite{S1} (see also \cite{Ko, Ko3}).  Shokurov himself established the $2$-dimensional case of the ACC conjecture using Mori's minimal model program \cite{S}. There has been remarkable recent progress on the ACC conjecture, and it is now established in dimension $3$ by work of Alexeev \cite{A}, and for smooth varieties in arbitrary dimension by de Fernex-Ein-Musta\c{t}\u{a} \cite{dFEM}.  Finally, the ACC was established in full generality by Hacon-McKernan-Xu \cite{HMX}. We refer the reader also to the earlier works \cite{Ko3, dFM} where important partial results were obtained. The proof of the smooth case of the Shokurov's ACC conjecture implies
\begin{thm}[de Fernex-Ein-Musta\c{t}\u{a} \cite{dFEM}, Theorem 1.1]
Conjecture \ref{conj: ACC} holds.
\end{thm}

In complex dimension $2$ there have been several contributions to Conjecture~\ref{conj: ACC}. Igusa \cite{I} computed the log-canoncial thresholds of irreducible plane curves with singularities at the origin.  This was subsequently improved by Kuwata \cite{K} who computed all of the critical integrability indices in dimension $2$ using algebraic techniques. Phong-Sturm, using analytic techniques they developed in \cite{PS}, described all the critical integrability indices in dimension $2$ \cite{PS1}, giving an analytic proof of Conjecture~\ref{conj: ACC}, and characterizing all the accumulation points of $\mathcal{C}(2)$.  Furthermore, their results hold more generally in the setting of real analytic functions in $\mathbb{R}^2$ \cite[Remark 1]{PS1}.   Favre-Jonsson proved Conjecture~\ref{conj: ACC} in complex dimension $2$ \cite{FJ} as a by product of their robust algebraic techniques based on valuations.  More recently, Hai-Hiep-Hung \cite{HHH} gave another analytic proof of Conjecture \ref{conj: ACC} in dimension $2$.  Summarizing;

\begin{thm}\label{thm: main}
Let $\mathcal{C}(2)$ denote the set of integrability indices for holomorphic (resp. real analytic) germs defined near the origin in $\mathbb{C}^2$ (resp. $\mathbb{R}^2$).  Then $\mathcal{C}(2)$ satisfies the ascending chain condition.  Furthermore, the accumulation points of $\mathcal{C}(2)$ are precisely $\{0\}\cup \mathcal{C}(1)$.
\end{thm}

In this generality, the Theorem~\ref{thm: main} is due to Phong-Sturm \cite[Theorem A, Remark 1]{PS1}.  The goal of this note is to give a new ``elementary" proof of Theorem~\ref{thm: main}.  The main idea in our approach is to use the connection between integrability indices and convex bodies.  Recall that the Newton polyhedron of $f$ is the convex polyhedron $NP(f) \subset \mathbb{R}^{2}_{\geq 0 }$ constructed as the convex hull of the set $(p,q)+ \mathbb{R}^{2}_{\geq0}$, where $x^py^q$ appears in the Taylor series of $f$ with non-zero coefficient (see Section~\ref{sec: main} for more details).  The Newton polyhedron is not coordinate invariant, but in dimension two one can find a holomorphic (resp. real analytic) change of coordinates so that $NP(f)$ computes $c_0(f)$.

\begin{thm}\label{thm: NewtDist}
In $\mathbb{C}^2$ (resp. $\mathbb{R}^2$), there exists a holomorphic (resp. real analytic)  change of coordinates of the form
\[
(\tilde{x}, \tilde{y}) = (x-Q(y), y) \quad \text { or } \quad (\tilde{x}, \tilde{y}) = (x, y-Q(x))
\]
so that
\[
c_0(f) = 2\delta_{NP} \qquad (\text{resp. } \delta_{NP} )
\]
where $(\delta_{NP}^{-1},\delta_{NP}^{-1}) \in \del NP(f)$ is the Newton distance of $f$ in the coordinates $(\tilde{x}, \tilde{y})$.
\end{thm}

Assuming Theorem~\ref{thm: NewtDist}, our proof of Theorem~\ref{thm: main} is based on the observation that the Newton polyhedron is characterized by ``discrete" data (namely the integral vertices), and so one should expect some rigidity for the possible values obtained by intersecting with the diagonal.  By Theorem~\ref{thm: NewtDist} this implies rigidity for integrability indices.

Theorem~\ref{thm: NewtDist} was proved by Var\v{c}henko \cite{V}.  The description of the coordinate change in the real setting is implicit in \cite[Lemma 3.6]{V}, and these techniques generalize to the complex case.  In the real case, Theorem~\ref{thm: NewtDist} appears explicitly in the work of Phong-Stein-Sturm \cite{PSS} as an application of their techniques for studying stability of certain oscillatory integrals.  In the complex case, a version of Theorem~\ref{thm: NewtDist} is proved in \cite{ABCLM}, omitting only the description of the change of variables.  The techniques of \cite{ABCLM} are algebraic, making use of the local topological zeta function of Denef-Loeser \cite{DL}.  We remark that the explicit description of the coordinate change will be important for our argument. In $\mathbb{R}^{n}, n \geq 3$, an analogue of Theorem~\ref{thm: NewtDist} was proved by the author with Greenleaf and Pramanik \cite{CGP} though crucially, one must allow coordinate changes involving fractional power series.  Since the result we need is not explicitly stated in the literature we shall give a self-contained proof of Theorem~\ref{thm: NewtDist} in the spirit of the methods developed by Phong-Stein-Sturm \cite{PSS}, and later work of the author with Greenleaf and Pramanik \cite{CGP}.  These techniques are based on an analytic resolution of singularities algorithm (building on work of Bierstone-Milman \cite{BM, BM1} and Parusi\'nski \cite{P1,P2}) and sharp estimates.  Furthermore, we view this as an opportunity to illustrate the techniques of \cite{CGP} in an example.
\bigskip

{\bf Acknowledgements}:  The author is grateful to Alexander Var\v{c}henko for some helpful comments on an earlier draft.

\section{Background and Proofs}\label{sec: main}

Suppose $f(x,y)$ is a holomorphic (resp. real analytic) function defined in a neighbourhood of $0 \in\mathbb{C}^2$ (resp. $\mathbb{R}^2$) and with $f(0,0)=0$.  We expand $f$ as a power series
\begin{equation}\label{eq: Tseries}
f(x,y) = \sum_{(p,q) \in \mathbb{N}^{2}} a_{p,q}x^py^q.
\end{equation}
\begin{defn}
The Newton polyhedron of $f$, denoted $NP(f)$ is the convex set defined as
\[
NP(f) = {\rm Convex Hull}(\left\{ (p,q) +\mathbb{R}^2_{\geq 0} :  a_{p,q} \ne 0 \right\}).
\]
\end{defn}


In words, for every monomial $x^{p}y^{q}$ appearing in the Taylor series of $f$, one attaches a copy of the positive orthant at $(p,q)$, and then takes the convex hull.  The result is a non-compact convex body with polyhedral boundary.

\begin{rk}\label{rk: NPinvar}
It is important to note that the Newton polyhedron is {\em not} independent of the choice of coordinates.  For example, the polynomial $f(x,y)= (y-x)^{N}$ has Newton polyhedron $\{p+q \geq N\}$, while $g(x,y) = y^{N}$ has Newton polyhedron $\{q\geq N\}$. But clearly $f = g(x,y-x)$.
\end{rk}
\begin{defn}
The Newton distance, denoted $\delta_{NP}$ is defined by 
\[
\delta_{NP}^{-1}(f) = \inf \left\{t\in \mathbb{R}_{\geq 0}: (t,t) \in NP(f) \right\}
\]
\end{defn}
That is $(\delta_{NP}^{-1}, \delta_{NP}^{-1})$ is the point where the diagonal line $\{p=q\}$ intersects $\del NP(f)$.  The Newton polyhedron plays a fundamental role in the study of the zeroes of $f$; we will expand upon some of these connections, but refer the reader to \cite{Ko2} for a more thorough discussion.  

We define the order of $f$ at $0$ to be $N$, denoted ${\rm ord}_0f=N$, if $N$ is the minimal homogeneous degree of a non-zero monomial appearing in \eqref{eq: Tseries}.  Equivalently, $N$ is the order of vanishing of $f$ restricted to a generic line in $\mathbb{C}^2$ or $\mathbb{R}^2$.  Therefore, up to making a linear change of coordinates, we can assume that
\[
f(0,y) = cy^{N} + O(y^{N+1}), \qquad c\ne 0.
\]
By the Weierstrass preparation theorem we can write
\[
f(x,y) = (unit) \cdot P(x,y)
\]
where $unit$ denotes a non-vanishing holomorphic (resp. real analytic) function, and $P(x,y)$ is a Weierstrass polynomial of degree $N$ in $y$;
\[
P(x,y) = y^{N} + \sum_{\ell=0}^{N-1}b_{\ell}(x)y^{\ell}
\]
with $b_{\ell}(x)$ holomorphic (resp. real analytic) functions vanishing at $x=0$.
\begin{rk}\label{rk: genericCoords}
We point out that, for a generic choice of coordinates we can also arrange that $f(x,0) = c'x^{N} + O(x^{N+1})$.  In particular, we can ensure that $f(x,y)$ can be written simultaneously as a Weierstrass polynomial of order $N$ in $x$ and a Weierstrass polynomial of order $N$ in $y$ (up to multiplication by a unit).  This will be a convenient choice of coordinates to make in Section~\ref{sec:mainthm}.
\end{rk}

For completeness, and since we will later have to allow changes of coordinates, we will compute the critical integrability index of a general holomorphic function $f$.  Up to discarding a non-vanishing holomorphic function, we can take $f$ to be of the form
\[
f(x,y) := x^{\alpha}y^{\beta}P(x,y)
\]
where $P(x,y)$ is a Weierstrass polynomial in $y$ of degree $N$ with $P(x,0)\not\equiv0$.  It is a classical fact that after possibly shrinking the neighbourhood of the origin we can factor
\begin{equation}\label{eq: prodPuis}
P(x,y) = \prod_{\nu=1}^{N}(y-\phi_\nu(x))
\end{equation}
where the $\phi_\nu$ are convergent Puiseux series solutions of the form
\[
\phi_\nu(x) = b_\nu x^{a_\nu} + O(x^{A_\nu})
\]
with $b_\nu \in \mathbb{C}^{*}$ and $a_\nu \in \mathbb{Q}_{>0}$, $A_\nu>a_\nu$; we refer the reader to \cite[Section 3]{PSt} and the references therein for a complete discussion of the convergence properties of Puiseux series.  Let 
\[
0<a_1< \cdots < a_k
\]
be the distinct leading exponents of the $\phi_{\nu}$, and let $F_1,\ldots, F_k$ be the non-vertical/horizontal faces of the Newton polyhedron of $f$, numbered from left to right. The leading exponents are determined by
\[
\frac{1}{a_i} = -\text{ the slope of } F_i.
\]
Furthermore, the number of solutions $\phi_{\nu}$ with a given leading order $a_i$ is nothing but the length of the line obtained by projecting $F_i$ onto the $q$-axis.  This discussion is somewhat backwards, since the Puiseux series solutions are in general constructed from the Newton polyhedron inductively \cite{Ko2}.

Let us introduce some notation.  Set
\[
S_{i} = \{ \phi_\nu(x) : \phi_\nu(x) = bx^{a_i} + \cdots \text{ for some } b\ne 0\}, \qquad m_i := \#S_i
\]
to be the collection of Puiseux series roots with leading order $a_i$.  For $0\leq i \leq k$, let
\begin{equation}\label{eq:ABd def}
A_i = \sum_{j\leq i} m_ja_j + \alpha, \qquad B_i = \sum_{j>i}m_j + \beta, \qquad \delta _i^{-1} := \frac{A_i+a_iB_i}{1+a_i}.
\end{equation}
A short computation shows that $(A_i, B_i)$ are the vertices of the Newton polyhedron, and $(\delta_i^{-1}, \delta_i^{-1})$ is the point of intersection of the diagonal with the prolongation of the face $F_i$.  We define the {\em main face} of $\del NP(f)$ to be the face or vertex of $\del NP(f)$ which meets the diagonal; in particular, the main face may not be a face at all.

For each pair of roots $\phi_{\mu}, \phi_{\nu}$ we can write
\[
\phi_{\mu} = b_{\mu}x^{a_{\mu}} + \cdots, \qquad \phi_{\mu}- \phi_{\nu} = b_{\mu\nu}x^{r_{\mu\nu}}+ \cdots,
\]
with $b_{\mu},b_{\mu\nu} \ne 0$, and we define the {\em order of contact} between $\phi_{\mu}$ and $\phi_{\nu}$ to be $r_{\mu\nu}$.  We define $r_{\mu\nu}=+\infty$ if $\phi_{\mu} - \phi_{\nu} \equiv 0$. In what follows it will be convenient to fix two constants $\epsilon, D$, which we define as
\[
\epsilon = \frac{3}{4}\min_{\mu}\min_{\nu: \phi_{\mu} \not\equiv \phi_{\nu}} \{ |b_{\mu}|, |b_{\mu\nu}|\}, \qquad  D = 2\max_{\mu, \nu} \{ |b_{\mu}|, |b_{\mu\nu}|\}.
\]

\begin{rk}\label{rk: constants}
The precise values of $\epsilon, D$ will not matter.  What does matter is an upper bound for $\epsilon$ and a lower bound for $D$, which we have fixed with the above definition.  That is, if $0< \tilde{\epsilon} <\epsilon < D < \tilde{D} < +\infty$, then our arguments work just as well with $\tilde{\epsilon}, \tilde{D}$.
\end{rk}

Roots with order of contact larger than their leading order play an important role in the estimates to follow, and so we will introduce the sets
\[
S_{i,\ell} = \{ \phi_\nu(x) : \phi_\nu(x) = b_{i\ell}x^{a_i} + \cdots \},  \qquad m_{i\ell} := \#S_{i,\ell}.
\]
Let
\begin{equation}\label{eq: ordersOfContact}
a_i<r_1 < r_2 < \cdots <r_{m-1} < r_{m} = +\infty
\end{equation}
be the list of contact orders between $\phi_{\mu}, \phi_{\nu} \in S_{i,\ell}$ (so that $r_1 >a_i$); note that we do not impose that $\phi_{\mu}, \phi_{\nu}$ are distinct.

\section{Proof of Theorem~\ref{thm: NewtDist}}\label{sec:NewtDist}

We will give the proof of Theorem~\ref{thm: NewtDist} in the complex case, though the reader can check that the argument works just as well in the real case (see \cite{PSS, CGP}).  Our goal in this section is to estimate the integral
\begin{equation}\label{eq: mainInt}
\int_{U} |f|^{-2c}
\end{equation}
where $U \ni 0$ is an open set that we can shrink as we please.  Note we have inserted a factor of $2$ for convenience (this is customary in the complex setting).  We first establish that $c_0(f) \leq 2\delta_{NP}$, by estimating the integral in a region which is ``far" from the roots of $P$.  

\begin{nota}
We will use the symbol $\preceq$ to denote ``less than or equal, up to multiplication by a positive constant", so $a \preceq b$ means $a \leq C b$ for some $C>0$.
\end{nota}

Fix $0< \eta \ll 1$ and for $1 \leq i \leq k-1$, consider the regions
\begin{equation}\label{eq: hollowReg}
\begin{aligned}
R_{0,1} &= \{ D|x|^{a_1} < |y| < \epsilon\}\times \{0 \leq |x| \leq \eta\}\\
R_{i, i+1} &= \{D|x|^{a_{i+1}}< |y| < \epsilon |x|^{a_i} \} \times \{0\leq |x| \leq \eta\} \qquad  \\
R_{k, \infty} &= \{ 0 < |y| < \epsilon|x|^{a_k} \} \times \{0 \leq |x| \leq \eta\}
\end{aligned}
\end{equation}
We will explain how to estimate the integral over $R_{i, i+1}$ in detail, but let us first state the results for $R_{0,1}$ and $R_{k,\infty}$ first.  The integral over $R_{0,1}$ is finite if and only if
\[
c\alpha<1\quad \text{ and } c<\frac{1+a_1}{(A_1+ a_1B_1)}
\]
while the integral over $R_{k,\infty}$ is finite if and only if
\[
c\beta<1 \quad \text{ and } c<\frac{1+a_k}{(A_k+ a_kB_k)}.
\]
We impose these conditions from now on.  To estimate the integral on $R_{i,i+1}$ it suffices to sharply estimate the quantity $|y-\phi_\nu(x)|$.  There are two cases.  First, if $\phi_\nu \in S_j$ for $j > i$ then we get the estimate
\begin{equation*}
\begin{aligned}
|y-\phi_\nu(x)| &\leq |y|+|\phi_\nu(x)| \leq \epsilon |x|^{a_i}  +\frac{D}{2} |x|^{a_j} \leq 2 \epsilon |x|^{a_i}\\
|y-\phi_\nu(x)| &\geq |y|-|\phi_\nu(x)| \geq |y| -\frac{D}{2} |x|^{a_j} \geq \frac{1}{2} |y|.
\end{aligned}
\end{equation*}
On the other hand, if $\phi_\nu \in S_{j}$ for $j\leq i$, then definition of $\epsilon$ gives
\begin{equation*}
\begin{aligned}
|y-\phi_\nu(x)| &\leq |y|+|\phi_\nu(x)| \leq \epsilon |x|^{a_i}  +\frac{D}{2} |x|^{a_j} \leq D |x|^{a_{j}}\\
|y-\phi_\nu(x)| &\geq |\phi_\nu(x)|-|y| \geq\frac{4\epsilon}{3} |x|^{a_j} - \epsilon |x|^{a_i} \geq \frac{1}{3} \epsilon |x|^{a_j}.
\end{aligned}
\end{equation*}
Thus, in $R_{i, i+1}$ we have
\begin{equation}\label{eq: est1Hol}
|x|^{A_i}|y|^{B_i} \preceq |f| \preceq |y|^{\beta}|x|^{A_i + a_i(B_i-\beta)}.
\end{equation}
We estimate, using $1>c\beta$
\[
\begin{aligned}
\int_{R_{i,i+1}} |f|^{-2c} &\succeq \int_{0<|x|<\eta} \int_{C|x|^{a_{i+1}}< |y| < \epsilon |x|^{a_i}} |y|^{1-2c\beta}|x|^{1-2c(A_i + a_i(B_i-\beta))}d|y| d|x|\\
 &\succeq \int_{0<|x|<\eta} |x|^{1 +2a_i-2c(A_i + a_iB_i)} d|x|\\
\end{aligned}
\]
from which we see that the integral over $R_{i,i+1}$ is infinite if
\[
c \geq \delta_i = \frac{1+a_i}{A_i+a_iB_i}.
\]
Summarizing, we have proved that the integral~\eqref{eq: mainInt} diverges if
\[
c \geq \min_{i}\left\{ \frac{1}{\alpha}, \frac{1}{\beta}, \delta_i \right\} = \delta_{NP}.
\]
Next we show that this bound is sharp on the regions~\eqref{eq: hollowReg}.  So suppose $c<\delta_{NP}$.  Using the estimate \eqref{eq: est1Hol}
\[
\begin{aligned}
\int_{R_{i,i+1}} |f|^{-2c} &\preceq \int_{\{0<|x|<\eta\}} \int_{C|x|^{a_{i+1}}< |y| < \epsilon |x|^{a_i}} |y|^{1-2cB_i}|x|^{1-2c A_i }d|y| d|x|\\
 &\preceq \int_{\{0<|x|<\eta\}} |x|^{1+2a_i-2c(A_i + a_iB_i)} d|x|\\
\end{aligned}
\]
 provided $1>cB_i$.  If instead $1<cB_i$, then we get
 \[
\int_{R_{i,i+1}} |f|^{-2c} \preceq \int_{\{0<|x|<\eta\}} |x|^{1+2a_{i+1}-2c(A_i + a_{i+1}B_i)} d|x|\\
\]
Since $A_i+ a_{i+1}B_i = A_{i+1}+a_{i+1}B_{i+1}$, we see that the integral over $R_{i,i+1}$ is finite if $c< \min\{ \delta_i, \delta_{i+1}\} \leq \delta_{NP}$.  The regions $R_{0,1}, R_{k,\infty}$ are estimated to similarly and we conclude;
\begin{lem}
The integral of $|f|^{-2c}$ over the regions in \eqref{eq: hollowReg} converges if and only if
\[
c < \min_{i}\left\{ \frac{1}{\alpha}, \frac{1}{\beta}, \delta_i \right\} = \delta_{NP}.
\]
In particular, $c_0(f) \leq 2\delta_{NP}$ for any coordinate system.
\end{lem}
\begin{rk}\label{rk: orderBd}
Note that if ${\rm ord}_0 f =N$ then we always have $c_0(f) \leq 2\delta_{NP} \leq 4N^{-1}$.  This bound can be derived in a more elementary way than this by converting to spherical coordinates; see \cite[Lemma 5.1]{PS}
\end{rk}

It remains to show that the integral~\eqref{eq: mainInt} is finite for $c < \delta_{NP}$.  This will be achieved by estimating the integral on the parts of $U$ not covered by the regions described in \eqref{eq: hollowReg}. For $1\leq i \leq k$ we write
\begin{equation}\label{eq: largeSolidHorn}
V_i = \{\epsilon |x|^{a_{i}}< |y| < D |x|^{a_i} \} \times \{0 \leq |x| \leq \eta \}.
\end{equation}
The $V_i$ together with the regions in~\eqref{eq: hollowReg} cover a neighbourhood of $0$.  The estimates on $V_i$ are somewhat more involved, owing to the fact that $V_i$ contains the roots in $S_i$.  As our estimates need to be essentially sharp, we must take care in how we decompose $V_i$, particularly with respect to isolating the roots of $f$.  The correct way to do this is by organizing the roots of $f$ according to their complexity, measured by the order of contact.  Let us focus only on the roots in $S_{i,\ell}$.

\begin{defn}
A cluster at level $r$ is a set $C_{r} \subset S_{i,\ell}$ such that, for all $\phi_\mu, \phi_\nu \in C_{r}$ we have
\[
\phi_\mu - \phi_\nu = b_{\mu\nu}x^{r_{\mu\nu}} + \cdots
\]
for $b_{\mu\nu}\ne 0$ and $r_{\mu\nu} \geq r$.  Furthermore, $C_{r}$ is maximal with this property in the sense that if $\phi_\mu \in C_{r}$, and $\phi_\nu \notin C_{r}$ then
\[
\phi_\mu - \phi_\nu = b_{\mu\nu}x^{r_{\mu\nu}} + \cdots
\]
for $b_{\mu\nu} \ne 0$ and $r_{\mu\nu} < r$.
\end{defn}
Note that if $\phi_\mu \in C_{r}$ and $\phi_\nu \notin C_{r}$, then the order of contact between $\phi_\mu$ and $\phi_\nu$ depends only on $C_r$ and not on the choice of $\phi_\mu$.
 
\begin{defn}
Let $C_r$ be a cluster at level $r$, and fix $\phi_{\mu}\in C_{r}$.  The maximal order of contact of $C_r$ is
\[
\gamma(C_r) := \max_{\nu} \{ r_{\mu\nu} :  \phi_\mu-\phi_\nu = b_{\mu\nu}x^{r_{\mu\nu}} + \cdots \quad  \text{ where } \quad \phi_\nu \notin C_{r} \}.
\]
Observe that $\gamma(C_r) <r$.  
\end{defn}
An important point is that each root $\phi_\nu \in S_{i}$ appears in a list of distinct clusters with decreasing order
\[
\{ \phi_\nu\} = C_{\infty}(\phi_{\nu}) \subsetneq C_{r_1} \subsetneq C_{r_2} \subsetneq\cdots \subsetneq S_{i,\ell} \subset S_i
\]
where $r_k = \gamma(C_{r_{k-1}})$.  Decomposing $V_i$ will involve the use of sets adapted to clusters; we call these sets ``horns".

\begin{defn}
A hollow horn about a cluster $C_r$ at level $r$ is
\[
\begin{aligned}
HH(C_r)&:= \{ |x|^{r} \preceq |y-C_r| \preceq |x|^r \}\\
&= \bigcap_{\phi_j \in C_r}\{ \epsilon |x|^r< |y-\phi_j| \leq D |x|^{r}\}.
\end{aligned}
\]
A solid horn about a cluster $C_r$ at level $r$ is
\[
\begin{aligned}
SH(C_r) &:= \{ |y-C_r| \preceq |x|^{\gamma(C_r)} \}\\
&= \bigcup_{\phi_j \in C_r}\{ |y-\phi_j| \leq \epsilon |x|^{\gamma(C_r)}\}.
\end{aligned}
\]
\end{defn}

\begin{prop}\label{prop: clusterint}
If $C_r \subset S_{i,\ell}\subset S_i$ is a cluster at level $r$, then
\[
\int_{SH(C_r)} |f|^{-2c} < +\infty
\]
provided
\begin{equation}\label{prop: cpropbnd}
c< \min\left\{m_{i\ell}^{-1}, \frac{1+a_i + (r_{m-1}-a_i)}{A_i +a_iB_i + m_{i\ell}(r_{m-1}-a_i)}, \delta_{NP} \right\}
\end{equation}
where $r_{m-1}$ is the maximal, non-infinite order of contact between roots in $S_{i,\ell}$ (see~\eqref{eq: ordersOfContact})
\end{prop}
\begin{proof}
Suppose $c$ satisfies the bound in~\eqref{prop: cpropbnd}.  The proof goes by induction on the level of the cluster,  beginning with clusters at level $\infty$.  By definition, a cluster at level $\infty$ is a set consisting of one of the distinct roots of $P$, counted with multiplicity.  We will estimate the integral over a solid horn about a cluster at level $\infty$.  To this end, fix $\phi_\nu \in S_{i,\ell}$, which we assume has maximal order of contact $r_\nu$.  A solid horn about the cluster $\{\phi_{\nu}\}$ consists of the region
\begin{equation}\label{eq:inftyhorn}
 \{ |y-\phi_{\nu}(x)| < \epsilon |x|^{r_{\nu}} \}.
\end{equation}
On this region we can estimate $|f|$ from below in the following way.  First, if $\phi_\mu \in S_j \cap S_{i,\ell}^c$, or if $\phi_{\mu} \equiv 0$, then
\begin{equation}\label{eq:estSj}
|y-\phi_\mu(x)| \geq |\phi_\mu-\phi_{\nu}|-|y-\phi_{\nu}| \geq\frac{\epsilon}{4} |x|^{a_j},
\end{equation}
where we used the definition of $\epsilon$.  On the other hand, if $\phi_\mu \in S_{i,\ell}$ but $\phi_{\mu} \ne \phi_{\nu}$ then we have the (possibly wasteful) estimate
\[
|y-\phi_\mu(x)| \geq |\phi_\mu-\phi_{\nu}|-|y-\phi_{\nu}| \geq\frac{\epsilon}{4} |x|^{r_\nu},
\]
again, using the definition of $\epsilon$.  Letting $d_{\nu}$ denote the multiplicity of $\phi_{\nu}$, or equivalently the cardinality of the cluster, we get the estimate
\[
|f| \geq |y-\phi_{\nu}|^{d_{\nu}}|x|^{A_i +a_iB_i + m_{i\ell}(r_{\nu}-a_i) -d_{\nu}r_{\nu}}.
\]
Thanks to the bound $c<m_{i\ell}^{-1}\leq d_{\nu}^{-1}$, integrating the estimate over the region in \eqref{eq:inftyhorn} we get 
\[
\int |f|^{-2c} \preceq \int_{0}^{\epsilon} |x|^{1+2r_{\nu} -2c(A_i +a_iB_i + m_{i\ell}(r_{\nu}-a_i))} d|x|
\]
and so we see that the integral converges if
\begin{equation}\label{eq:estmaxcont}
c < \frac{1+a_i + (r_{\nu}-a_i)}{A_i +a_iB_i + m_{i\ell}(r_{\nu}-a_i)}.
\end{equation}
We state the following trivial lemma, for convenience
\begin{lem}\label{lem: dumb}  
Let
\[
c(x) =\frac{1+a_i + x}{A_i +a_iB_i + m_{i\ell}x}.
\]
Then $c(x) \geq c(0)$ for $x>0$ if and only if $m_{i\ell} \leq c(0)^{-1}$.  If $m_{i\ell} > c(0)^{-1}$, then $c(x)$ is an increasing function of $x>0$.
\end{lem}

By the lemma, we see that the inequality in \eqref{eq:estmaxcont} is implied by the inequality in Proposition~\ref{prop: clusterint}, and hence establishes the proposition for clusters at level $\infty$.

Now, suppose we have proved Proposition~\ref{prop: clusterint} for solid horns about clusters at level $r \geq r_{s+1}$.  Let $C_{r_s}$ be a cluster at level $r_s$.  To ease notation, let us denote $r=r_s$.  Our goal is to estimate the integral over the solid horn $SH(C_{r})$.  The key point is that $C_{r}$ can be written as a disjoint union of clusters $C_{\hat{r}} \subsetneq C_{r}$ each having $\gamma(C_{\hat{r}}) = r$.  This allows us to decompose 
\[
SH(C_r) = E_1 \sqcup E_2\sqcup E_3
\]
where,
\[
\begin{aligned}
E_1 &= \bigcup_{C_{\hat{r}}\subsetneq C_r} SH(C_{\hat{r}})\\
E_2 &= HH(C_r)\\
E_3 &= \bigcup_{\phi_j \in C_{r}} \{ D |x|^{r} \leq |y-\phi_j| < \epsilon |x|^{\gamma(C_r)}\}.
\end{aligned}
\]
the first union being taking over the $C_{\hat{r}}$ described above.  Since we have already established the estimate on the solid horns $SH(C_{\hat{r}})$ for $\hat{r}>r$, it suffices to estimate the integral on $E_2$ and $E_3$. Let us consider the estimate on $E_2$.  First, if $\phi_\mu \in S_j \cap S_{i,\ell}^c$, or $\phi_{\mu} \equiv 0$, then \eqref{eq:estSj} still holds, so it suffices to estimate $|y-\phi_\mu|$ when $\phi_{\mu} \in S_{i,\ell}$.  There are then two cases, depending on whether $\phi_{\mu} \in C_r$ or $\phi_{\mu}\in S_{i,\ell}\cap C_r^c$.  If $\phi_{\mu} \in C_{r}$, then from the definition of $HH(C_r)$ we have
\[
|y-\phi_{\mu}| \succeq |x|^{r}.
\]
If instead $\phi_{\mu} \in S_{i,\ell} \cap C_{r}^{c}$ then for any $\phi_{j} \in C_{r}$ we have
\[
|y-\phi_{\mu}| \geq |\phi_{j} - \phi_{\mu}| - |\phi_{j}- y| \geq \epsilon|x|^{\gamma(C_r)} - D|x|^{r} \succeq |x|^{r},
\]
and so we obtain the estimate
\[
|f| \succeq |x|^{A_i+a_iB_i + m_{ik}(r-a_i)}.
\]
Integrating this estimate over the hollow horn $HH(C_r)$ yields the bound
\[
c < \frac{1+a_i + (r-a_i)}{A_i +a_iB_i + m_{i\ell}(r-a_i)}.
\]
Since $r\leq r_{m-1}$, Lemma~\ref{lem: dumb} shows that this bound is implied by the assumptions of Proposition~\ref{prop: clusterint}.  It remains to estimate the integral over the region $E_{3}$. On this region the estimate \eqref{eq:estSj} still holds, so it suffices to estimate $|y-\phi_\mu|$ when $\phi_{\mu} \in S_{i,\ell}$.  We consider separately each of the regions
 \[
 U_j:= \{ D |x|^{r} \leq |y-\phi_j| < \epsilon|x|^{\gamma(C_r)}\}.
 \]
For any  $\phi_{\mu} \in C_{r}$, different from $\phi_{j}$ we have the estimate
\[
|y-\phi_{\mu}| \geq |y-\phi_j| - |\phi_j - \phi_{\mu}| \geq \frac{1}{2}|y-\phi_j|
\]
using the definitions of $D, U_j$.  On the other hand, if $\phi_{\mu}\in S_{i,\ell}\cap C_{r}^c$ then we have the estimate
\[
|y-\phi_{\mu}| \geq |\phi_j - \phi_{\mu}| - |y-\phi_j| \geq \frac{\epsilon}{4}|x|^{\gamma(C_r)}.
\]
Combining these estimates gives
\begin{equation}\label{eq: genEst}
|f| \geq |y-\phi_{j}|^{\#C_{r}}|x|^{A_i +a_iB_i + m_{i,\ell}(\gamma(C_r)-a_i)- \#C_r \cdot \gamma(C_r)}.
\end{equation}
The reader may note the similarity between this estimate and the one obtained in the initial step of the induction (where $r=+\infty$).  Again, thanks to the bound $\#C_r \leq m_{i\ell} < c^{-1}$, integrating~\eqref{eq: genEst} over the region $U_{j}$ gives the bound
\[
c < \frac{1+a_i + (\gamma(C_r)-a_i)}{A_i +a_iB_i + m_{i\ell}(\gamma(C_r)-a_i)}.
\]
Appealing again to Lemma~\ref{lem: dumb} establishes Proposition~\ref{prop: clusterint}. 
\end{proof}

Let's assume now that $m_{i\ell}\leq \delta_{NP}^{-1}$ for all $i,\ell$ and finish the proof of Theorem~\ref{thm: NewtDist}.  In this case, Proposition~\ref{prop: clusterint}, combined with Lemma~\ref{lem: dumb} implies that the integral of $|f|^{-2c}$ over solid horns about clusters $C_r\subset S_{i,\ell}$ is finite provided $c<\delta_{NP}$.  Up to decreasing $\epsilon$ and increasing $D$ (see Remark~\ref{rk: constants}), it is easy to see that the region $V_i$, defined in~\eqref{eq: largeSolidHorn}, can be decomposed into a union of solid horns and hollow horns about the $S_{i,\ell}$.  Since the $S_{i,\ell}$ are themselves clusters, the estimates on solid horns about $S_{i,\ell}$ follow from Proposition~\ref{prop: clusterint}.  The estimates on the hollow horns about the $S_{i,\ell}$ are identical to those already obtained.  Alternatively, one can start from the very beginning, considering the $S_i$ themselves as clusters and running the same induction.  We have

\begin{prop}\label{prop: almostNewtDist}
In $\mathbb{C}^2$ (resp. $\mathbb{R}^2$), if $m_{i\ell} \leq \delta_{NP}^{-1}$ for all $i,\ell$, then we have
\[
c_0(f) = 2\delta_{NP} \qquad (\text{resp. } \delta_{NP} .)
\]
\end{prop}

Theorem~\ref{thm: NewtDist} now follows from Proposition~\ref{prop: COV} below.  Before embarking on the proof, the reader may consider the examples mentioned in Remark~\ref{rk: NPinvar}.  In that case the function $f = (y-x)^{N}$ does not satisfy the assumptions of Proposition~\ref{prop: almostNewtDist}.  However, $f$ can be transformed to the function $g = y^{N}$, to which Proposition~\ref{prop: almostNewtDist} does apply,  by a holomorphic change of variables.

\begin{prop}[Phong-Stein-Sturm]\label{prop: COV}
In the above notation, we have
\[
m_{i\ell} \leq \delta_{NP}^{-1}
\]
unless $m_{i\ell}$ corresponds to the main face of $\del NP(f)$.  If $m_{i\ell} > \delta_{NP}^{-1}$ then there exists a holomorphic function $Q(\cdot)$ so that after making the change of coordinates
\[
(\tilde{x}, \tilde{y}) = (x-Q(y), y) ,\quad \text { or } (\tilde{x}, \tilde{y}) = (x, y-Q(x))
\]
the critical integrability index $c_0(f)$ is computed by the Newton distance of $f(\tilde{x}, \tilde{y})$.  
\end{prop}
\begin{proof}
We sketch a proof, following Phong-Stein-Sturm \cite[Theorem 5]{PSS}.  Let $\pi_{p}, \pi_{q}$ be the projections to the $p$ and $q$-axes respectively.  For every face $F_i$ of $NP(f)$ let
\[
A_i = \text{ Length of } \pi_{q}(F_i), \qquad B_i = \text{ Length of } \pi_{p}(F_i)
\]
Recall that $A_i$ is the cardinality of the set of roots with leading order determined by the slope of $F_i$, so if $F_i$ is right of the main face, then $m_{i\ell} \leq A_i \leq \delta_{NP}$.  Now suppose we have a solution
\[
\phi_{\nu} = bx^{p_{\nu}/q_{\nu}} +\cdots
\]
with $p_\nu, q_\nu$ relatively prime, and corresponding to a face $F_i$.  If $\zeta$ is a $q_{\nu}$-th root of unity then it is easy to see that
\[
\tilde{\phi}_{\nu} = \zeta bx^{p_{\nu}/q_{\nu}} +\cdots
\]
is also a root of $f$, and therefore $m_{i\ell} \leq A_i/q_{\nu}$. Since $A_i/B_i =q_\nu/p_\nu$, we can further conclude
\[
A_i \geq q_{\nu}m_{i\ell} \qquad B_{i} \geq p_{\nu}m_{i\ell}.
\]
Now if $F_i$ is left of the main face, then we fhave $B_i \leq \delta_{NP}$, and so $m_{i\ell}\leq \delta_{NP}$.  If the main face is a vertex we're done, so we can assume that is not the case.   Let $F_i$ denote the main face.  $F_i$ is cutout by the line
\[
L(p,q) := \frac{qp_{\nu} + q_{\nu}p}{p_{\nu}+ q_{\nu}}- \delta_{NP}^{-1} = 0,
\]
and the point $(0, A_i)$ lies in the region $\{ L \leq 0\}$.  Writing this out gives
\[
\delta_{NP}^{-1} \geq \frac{A_ip_{\nu}}{p_{\nu}+ q_{\nu}} \geq m_{i\ell}\frac{p_{\nu}q_{\nu}}{p_{\nu}+ q_{\nu}}.
\]
If $m_{i\ell} > \delta_{NP}^{-1}$, then $p_{\nu} + q_{\nu} > p_{\nu}q_{\nu}$ so either $p_{\nu}$ or $q_{\nu}$ must be $1$.  It remains only to construct the change of variables.  Let us assume $q_{\nu}=1$ for simplicity, otherwise we write the roots as functions of $y$ and the argument is the same.  By assumption we have $m_{i\ell} > \delta_{NP}^{-1}$ roots of the form
\[
\phi_{\nu} = bx^{p_{\nu}} +\cdots.
\]
We take $Q(x) = bx^{p_{\nu}} +\cdots$ to be the powerseries, or polynomial of maximal degree so that there are more than $\delta_{NP}^{-1}$ roots $\phi_{\nu}$ with
\[
\phi_{\nu} = Q(x) + \text{ higher order terms}.
\]
It is then not hard to check that in the new coordinates $(\tilde{x}, \tilde{y}) = (x, y-Q(x))$ we have $\tilde{\delta}_{NP} < \delta_{NP}$ and $m_{i\ell} \leq \tilde{\delta}_{NP}^{-1}$ for every $i,\ell$.  An equivalent approach is to construct $Q(x)$ inductively, beginning with the coordinate transformation $\phi(x,y) = (x,y-bx^{p_{\nu}})$, and examining the change in the Newton polygon.  We leave the details to the reader.
\end{proof}

\section{Proof of Theorem~\ref{thm: main}}\label{sec:mainthm}

Let $\{f_n\}$ be a sequence of holomorphic (or real analytic) functions with $c_0(f_n)$ converging to $c_{\infty} > 0$.  By Remark~\ref{rk: orderBd}, up to passing to a subsequence, we can assume that the ${\rm ord}_{0}f_n = N$ for all $n$.  By Remark~\ref{rk: genericCoords}, after making a generic linear change of coordinates we can assume that each $f_n$ has ${\rm ord}_0f_{n}(x,0) = {\rm ord}_0f_n(0,y) = N$, and that $f_n$ can be written simultaneously as a Weierstrass polynomial of order $N$ in $x$ and $y$.  By Theorem~\ref{thm: NewtDist}, for each $n\in \mathbb{N}$, we can make a change of coordinates of the type described in Proposition~\ref{prop: COV} so that $c_0(f_n)$ is computed by the Newton distance $\delta_{NP}(f_n)$.  Suppose that the $n$-th change of variables is of the form
 \[
 (\tilde{x},\tilde{y}) = (x, y-Q_n(x))
 \]
 Then we can write
 \begin{equation}\label{eq:WPrep3}
 f_{n}(\tilde{x},\tilde{y}) = (unit)\cdot \left( \tilde{y}^N + \sum_{i=0}^{N-1}\tilde{b}_{\ell}(\tilde{x})\tilde{y}^\ell\right)
 \end{equation}
 with $\tilde{b}_{\ell}(\tilde{x})$ holomorphic (resp. real analytic) and vanishing at the origin.  The same thing holds in the case that $ (\tilde{x},\tilde{y}) = (x-Q_n(y), y)$, using the representation of $f_n$ as a Weierstrass polynomial of degree $N$ in $x$.  By renaming the variables we can always assume that there are coordinates where $\delta_{NP}(f_n)$ computes $c_0(f_n)$ and in these coordinates $f_n$ has the form~\eqref{eq:WPrep3}.  In particular, each Newton polygon $NP(f_n)$ has a vertex at the point $(0,N)$, and so contains the set $(0,N) + \mathbb{R}^2_{\geq 0}$.  Every vertex of $\del NP(f_n)$ lying to the left of the diagonal must be an element of the set 
\[
\mathcal{L}:= \{(p,q)\in \mathbb{N}^{2} : q\geq p, \quad q\leq N\}.
\]
Since $\mathcal{L}$ is finite, by passing to a subsequence, we can assume that $NP(f_n) \cap \mathcal{L}$ is independent of $n$.  Let $(p^{*},q^{*})$ be the vertex of $\del NP(f_n)$ lying in $\mathcal{L}$ and with $p^{*}$ maximal-- that is, $(p^{*},q^{*})$ is the first vertex of $\del NP(f_n)$ lying on, or left of, the diagonal. There are then two cases (we assume the complex case now, to fix constants).

If $p^{*} =q^{*}$, then $c_0(f_n) =2\delta_{NP}(f_n) = 2/p^{*}$ and we're finished, so we can assume $p^{*}< q^{*}$.  For simplicity, we now consider the increasing and decreasing cases separately, though the principle is the same.

{\bf Increasing case:}  Let $(p_n,q_n)$ be the left most vertex of $NP(f_n)$ lying in 
\[
\mathcal{R}=\{(p,q)\in \mathbb{N}^{2} : q< p\}.
\]
That is $(p_n, q_n)$ is the first vertex lying strictly to the right of the diagonal.  Since $p^{*}<q^{*}$, we have that $2c_0(f_n)^{-1} = \delta_{NP}(f_n)^{-1}$ lies in the interior of the line connecting $(p^{*}, q^{*})$ to $(p_n, q_n)$.  Moreover, thanks to the fact that $c_0(f_n)$ is increasing, we know that $(p_n,q_n)$ must lie in the region of $\mathcal{R}$ lying below the line connecting $(p^{*}, q^{*})$ to $(p_1,q_1)$.  But this region is compact and hence there are only finitely many possible choice of $(p_n, q_n)$. In particular, $\delta_{NP}(f_n)$ must stabilize.

{\bf Decreasing case:}  Essentially the same argument works in the decreasing case.  Assume that 
\[
c_0(f_n) \searrow c_{\infty}.
\]
If $c_{\infty} = \frac{2}{q^{*}}$ then we're done, since $2/q^{*} \in \mathcal{C}(1)$, so we can assume that $c_{\infty} >2/q^{*}$.  We now consider the line $L$ through $(p^{*}, q^{*})$ and $(c_{\infty}^{-1}, c_{\infty}^{-1})$.  Since $2c_{\infty}^{-1} < q^{*}$, $L$ must intersect the $p$-axis at some point.  Since $c_0(f_n)$ is decreasing, the points $(p_n, q_n)$ lie in $\mathcal{R}$, but below $L$.  This set is bounded, and hence contains only finitely many lattice points.  It follows that $c_0(f_n)$ must stabilize.  It remains only to demonstrate that points in $\mathcal{C}(1)$ are actually accumulation points which can be deduced by applying Proposition~\ref{prop: almostNewtDist} to the functions $f_{n} = y^{n}-x^{m}$ and letting $n\rightarrow \infty$.

\end{document}